\documentclass{amsart}

\usepackage{mathrsfs}
\usepackage{epsf, graphicx}
\usepackage{latexsym,amsfonts,amsbsy,amssymb}
\usepackage{subfigure}
\usepackage{multirow}
\usepackage{amsmath}
\usepackage{color}

\usepackage{lipsum}
\usepackage{amsfonts}
\usepackage{graphicx}
\usepackage{epstopdf}
\usepackage{color}
\usepackage{subfigure}
\usepackage{amssymb}
\usepackage{amsmath}
\usepackage{multirow}
\usepackage{enumerate}
\usepackage{cite,stmaryrd,txfonts}
\usepackage{tikz}
\usepackage{algorithm}

\input{undertilde}
\usetikzlibrary{snakes}

\usepackage{verbatim}

\usepackage{epic}

\usepackage{empheq}

\usepackage{ulem}

\newtheorem{theorem}{Theorem}

\newtheorem{lemma}[theorem]{Lemma}

\begin{document}
\title{A lowest-order mixed finite element method for the elastic transmission eigenvalue problem}

\author{Yingxia Xi}
\address{School of Science, Nanjing University of Science and Technology, Nanjing 210094, People's Republic of China}
\email{xiyingxia@njust.edu.cn}
\thanks{The research of Y. Xi is supported in part by Start-up Fund for Scientific Research, Nanjing University of Science and Technology (No. AE89991/109).}

\author{Xia Ji}
\address{LSEC, Institute of Computational Mathematics and Scientific/Engineering Computing, Academy of Mathematics and System Sciences, Chinese Academy of Sciences, Beijing 100190, People's Republic of China}
\email{jixia@lsec.cc.ac.cn}
\thanks{The research of X. Ji is partially supported by the National Natural Science Foundation of China with Grant Nos. 11271018 and 91630313, and National Center for Mathematics and Interdisciplinary Sciences, Chinese Academy of
Sciences.}
\subjclass[2000]{65N25,65N30,47B07}
\keywords{Transmission eigenvalue problem, Elastic wave equation, Mixed finite element method}

\begin{abstract}
The goal of this paper is to develop numerical methods computing a few smallest elastic interior
transmission eigenvalues, which are of practical importance in inverse elastic scattering theory. The problem is
challenging since it is nonlinear, non-self-adjoint, and of fourth order. In this paper, we construct a lowest-order mixed finite element method which is close to the Ciarlet-Raviart mixed finite element method. This scheme is based on Lagrange finite elements and is one of the less expensive methods in terms of the amount of degrees of freedom.
Due to the non-self-adjointness, the discretization of elastic transmission eigenvalue problem leads to a non-classical mixed method which does not fit into the framework of classical theoretical analysis. In stead, we obtain the convergence analysis based on the spectral approximation theory of compact operators.
Numerical examples are presented to verify the theory. Both real and complex eigenvalues can be obtained.
\end{abstract}

\maketitle

\section{Introduction}
Transmission eigenvalue problem is very important in the qualitative reconstruction in the inverse scattering theory of inhomogeneous media. For example, the eigenvalues can be used to estimate the physical properties of scattering object \cite{CakoniEtal2010IP, Sun2011IP}. The transmission eigenvalue problem is non-selfadjoint and is not covered by the standard theory of partial differential equations. It is numerically challenging because of the nonlinearity and the complicated spectral without a priori information. In most cases, the continuous problem is degenerate with an infinite dimensional eigenspace associated with the zero eigenvalue, which has no physical meaning and makes it difficult to be solved. There are different types of transmission eigenvalue problems, such as acoustic transmission eigenvalue problem, electromagnetic transmission eigenvalue problem, elastic transmission eigenvalue problem, etc.

Since 2010, effective numerical methods for the acoustic transmission eigenvalues have been developed by many researchers \cite{ColtonMonkSun2010, Sun2011SIAMNA,JiSunTurner2012ACMTOM, AnShen2013JSC, Geng2016, Kleefeld2013IP, JiSunXie2014JSC,CakoniMonkSun2014CMAM, LiEtal2014JSC, YangBiLiHan2016,
Camano2018,HuangEtal2016JCP,HuangEtal2017arXiv,SunZhou2016,XiJiZhang2018}. There are also much fewer works for the electromagnetic
transmission eigenvalue problem \cite{HuangHuangLin2015SIAMSC,MonkSun2012SIAMSC,SunXu2013IP}.
The goal of this paper is to develop effective numerical methods for transmission eigenvalue problem of elastic waves. Compared with the acoustic transmission eigenvalue problem, the eigenfunctions are vectors which make it more difficult to design convergent methods. There exist very limited numerical methods for elastic transmission eigenvalue problem,
To the best of our knowledge, there are only two works on numerical algorithm.
In \cite{Ji2017}, the elastic transmission eigenvalue problem is reformulated as the combination of a nonlinear function and a series of fourth order self-adjoint eigenvalue problems.
The nonlinear function values correspond to generalized eigenvalues of fourth order self-adjoint eigenvalue problems which can be discretized by $H^2$ conforming finite element methods.
The roots of the nonlinear function are the transmission eigenvalues. The authors apply the secant iterative method to compute the transmission eigenvalues. However, at each step, a fourth-order eigenvalue problem needs to be
solved and only real eigenvalues can be captured.
In \cite{XiJi2018}, an interior penalty discontinuous Galerkin method using C$^0$ Lagrange elements (C$^0$IP) is proposed for the elastic transmission eigenvalue problem.
They are simpler than $C^1$ elements and come in a natural hierarchy. It's much easier to be implemented. However, this method needs two sets of degrees of freedom at the common edge of adjacent grid cells. When the polynomial degree $p$ increases,  the degrees of freedom increase remarkably. Although the existence of transmission eigenvalues
is beyond our concern, we want to remark that there exist only a few studies on the existence of the elasticity transmission eigenvalue
problem \cite{Charalambopoulos2002JE, CharalambopoulosAnagnostopoulos2002JE,
BellisGuzina2010JE, BellisCakoniGuzina2013IMAAM}.  We hope that the numerical results can give some hints on the analysis of the elasticity transmission eigenvalue
problem.

In this paper, we construct a mixed finite element method for elastic transmission eigenvalue problem. For acoustic transmission eigenvalue problem, the related works for mixed element method can be referred to \cite{Camano2018,ColtonMonkSun2010,JiSunTurner2012ACMTOM,XiJiZhang2018,YangBiLiHan2016}.
The mixed scheme in \cite{Camano2018, JiSunTurner2012ACMTOM} which is close to the Ciarlet-Raviart discretization of biharmonic problems is based on Lagrange finite element method. For the nonzero transmission eigenvalues, this scheme is equivalent to the one proposed in \cite{ColtonMonkSun2010}. However, the scheme in \cite{Camano2018,JiSunTurner2012ACMTOM} can eliminate the zero transmission eigenvalues which has an infinite dimensional space and has no physical meaning. The mixed formulation in terms of three scaler fields and a spectral-mixed method is constructed in \cite{YangBiLiHan2016}. In \cite{XiJiZhang2018}, the authors propose a multi-level mixed formulation in terms of seven scaler fields. An equivalent linear mixed formulation of transmissoin eigenvalue problem which doesn't produce spurious modes even on non-convex domains is constructed.
The proposed scheme admits a natural nested discretization, based on that a multi-level
scheme is built. Optimal convergence rate and optimal computational cost can be obtained.

The mixed scheme for elastic transmission eigenvalue problem proposed in this paper also has similiarity to Ciarlet-Raviart discretization of biharmonic problems. This scheme is based on Lagrange finite elements and is one of the less expensive methods in terms of the amount of degrees of freedom. Besides, the proposed mixed scheme can eliminate the zero transmission eigenvalue. Because of the non-self-adjointness and non-linearity, elastic transmission eigenvalue problem leads to the non-classical mixed method which is not covered by standard theoretical analysis of mixed element method (the detailed description referred to Section \ref{ErrorEstimate}). Here we presented the convergence analysis under the framework of spectral approximation theory of compact operators \cite{BabuskaOsborn1991,Osborn1975MC} and the error analysis of a mixed finite element method for solving the Stokes problem \cite{Girault1986}.

The rest of this paper is organized as follows. In Section \ref{ETEP}, we introduce the problem, the mixed formulation and the variational formula for elastic transmission eigenvalue problem. In Section \ref{ErrorEstimate}, we introduce the solution operator and analyze the well-posedness of the operator. The discretization scheme is also presented and the convergence is proved. Numerical examples are presented in Section \ref{Experiments}.

\section{The elasticity transmission eigenvalue problem}\label{ETEP}
We begin with the notations used throughout this paper. All vectors will be denoted in bold script. Let $\boldsymbol x=(x, y)^\top\in\mathbb R^2$, $D\subset\mathbb{R}^2$ be a bounded convex Lipschitz domain, $\boldsymbol{u}(\boldsymbol{x})=(u_1(\boldsymbol{x}), u_2(\boldsymbol{x}))^\top$ be the displacement vector of the wave field and $\nabla\boldsymbol{u}$ be the displacement gradient tensor
\[
 \nabla\boldsymbol{u}=\begin{bmatrix}
                       \partial_x u_1 & \partial_y u_1\\
                       \partial_x u_2 & \partial_y u_2
                      \end{bmatrix}.
\] The strain tensor $\varepsilon(\boldsymbol{u})$ is given by
\[
\varepsilon(\boldsymbol{u})=\frac{1}{2}(\nabla\boldsymbol{u}
+(\nabla\boldsymbol{u})^\top),
\]
and the generalized Hooke law gives the stress tensor $\sigma(\boldsymbol{u})$
\[
 \sigma(\boldsymbol{u})=2\mu\varepsilon(\boldsymbol{u})+\lambda {\rm
tr}(\varepsilon(\boldsymbol{u})){\rm I},
\]
where the Lam\'{e} parameters $\mu, \lambda$ are two constants satisfying
$\mu>0, \lambda+\mu>0$, and ${\rm I}\in\mathbb{R}^{2\times 2}$ is the identity
matrix. Writing the above equation out, we have
\begin{equation}\label{DefSigma}
 \sigma(\boldsymbol{u})=\begin{bmatrix}
  (\lambda+2\mu)\partial_x u_1 + \lambda \partial_y u_2 & \mu (\partial_y u_1 +
\partial_x u_2)\\
\mu (\partial_x u_2 + \partial_y u_1) & \lambda\partial_x u_1 +
(\lambda+2\mu)\partial_y u_2
 \end{bmatrix}.
\end{equation}

The reduced Navier equation describes the two-dimensional elastic wave problem: Find
$\boldsymbol{u}$ with zero trace on $\partial D$, such that
\begin{equation}\label{ElasticityLHS1}
 \nabla\cdot\sigma(\boldsymbol{u})+\omega^2 \rho \boldsymbol{u}={\boldsymbol 0},
\quad\text{in}~D \subset \mathbb{R}^2,
\end{equation}
where $\omega>0$ is the angular frequency and $\rho$ is the mass density.




Now we are ready to give the definition of elastic transmission eigenvalue problem. Let $\mu_0, \lambda_0$ be the Lam\'{e} parameters of the free space
and the domain $D$ be a
homogeneous and isotropic elastic medium with Lam\'{e}
constants $\lambda_1$ and $\mu_1$. The transmission eigenvalue problem for the
elastic waves is: Find $\omega^2\neq0$ such that there exist non-trivial
solutions $\boldsymbol{u}, \boldsymbol{v}$ satisfying
\begin{equation}\label{tep}
\left\{
\begin{array}{rcl}
\nabla\cdot\sigma_0(\boldsymbol{u})+\omega^2 \rho_0 \boldsymbol{u}={\boldsymbol 0} &\quad\text{in} ~
D,\\
\nabla\cdot\sigma_1(\boldsymbol{v})+\omega^2 \rho_1 \boldsymbol{v}={\boldsymbol 0} &\quad\text{in} ~
D,\\
\boldsymbol{u}=\boldsymbol{v} &\quad\text{on} ~ \partial D,\\
\sigma_0(\boldsymbol{u}){\boldsymbol \nu}=\sigma_1(\boldsymbol{v}){\boldsymbol \nu} &\quad\text{on} ~
\partial D,
\end{array}
\right.
\end{equation}

where
\[
 \sigma_i(\boldsymbol{u})=\begin{bmatrix}
  (\lambda_i+2\mu_i)\partial_x u_1 + \lambda_i \partial_y u_2 & \mu_i
(\partial_y u_1 + \partial_x u_2)\\
\mu_i (\partial_x u_2 + \partial_y u_1) & \lambda_i\partial_x u_1 +
(\lambda_i+2\mu_i)\partial_y u_2
 \end{bmatrix},\quad i=0,1,
\]
and $\sigma {\boldsymbol \nu}$ denotes the matrix multiplication of the
stress tensor $\sigma$ and the unit outward normal $\boldsymbol \nu$.

In this paper,  the case of equal elastic
tensors \cite{BellisCakoniGuzina2013IMAAM}, i.e., $\rho_0 \ne \rho_1, \sigma_0 =
\sigma_1=\sigma$ is considered. In addition, we assume the following inequalities  for the mass
density distributions
\begin{equation}\label{pP}
q \le \rho_0({\boldsymbol x}) \le Q, \quad q_* \le
\rho_1({\boldsymbol x}) \le Q_*, \quad {\boldsymbol x} \in D,
\end{equation}
where $q, q_*$ and $Q, Q_*$ are positive constants and also assume that the two density distributions are "non-intersecting"\cite{BellisCakoniGuzina2013IMAAM}, i.e.
\begin{equation}
Q\leq1\leq q_*  \ \ \ \ or \ \ \ \ Q_*\leq1\leq q  \nonumber.
\end{equation}
We discuss the case $Q_*\leq1\leq q$ for illustration, and the case $Q\leq1\leq q_*$ is analogous.
Furthermore, denote by
\begin{eqnarray}  \nonumber
&\rho_{min}:=\min\limits_{{\boldsymbol x}\in D}(\rho_0({\boldsymbol x})-\rho_1({\boldsymbol x}))>0,& \\ \nonumber
&\rho_{max}:=\max\limits_{{\boldsymbol x}\in D}(\rho_0({\boldsymbol x})-\rho_1({\boldsymbol x})).& \\ \nonumber
\end{eqnarray}
It should be noted that we need $\rho_{min}>0$.

Define the Sobolev space
$$V=\{\boldsymbol{\phi} \in (H^2(D))^2: {\boldsymbol \phi} = {\boldsymbol 0} \text{
and } \sigma({\boldsymbol \phi}) {\boldsymbol \nu} = {\bf 0} ~ \text{on}~ \partial D \}.$$
Introducing the new variables $\boldsymbol{w}=\boldsymbol{u}-\boldsymbol{v}\in V$ and $\boldsymbol{p}=-\omega^2\boldsymbol{v}$, the system (\ref{tep})
can be written as
\begin{equation}\label{Problem}
\left\{
\begin{array}{rcll}
-\nabla\cdot\sigma(\boldsymbol{w})+(\rho_0-\rho_1)\boldsymbol{p}&=&\omega^2\rho_0\boldsymbol{w},&{\rm in}\ D,\\
\nabla\cdot\sigma(\boldsymbol{p})+\omega^2\rho_1\boldsymbol{p}&=&\boldsymbol{0}, &{\rm in}\ D.
\end{array}
\right.
\end{equation}
Further, dividing by $(\rho_0-\rho_1)$ and taking $\nabla\cdot\sigma$ in the first equation of (\ref{Problem}), we obtain $\boldsymbol{w}$ satisfying
\begin{equation}
-\nabla\cdot\sigma\left(\frac{\nabla\cdot\sigma(\boldsymbol{w})}{\rho_0-\rho_1}\right)+\nabla\cdot\sigma(\boldsymbol{p})=\omega^2\nabla\cdot\sigma(\frac{\rho_0\boldsymbol{w}}{\rho_0-\rho_1}).
\end{equation}
Then, using the second equation of (\ref{Problem}), $\boldsymbol{w}$ satisfies the equation
\begin{equation}\label{Fourth_regularity}
\nabla\cdot\sigma\left(\frac{\nabla\cdot\sigma(\boldsymbol{w})}{\rho_0-\rho_1}\right)+\omega^2\rho_1\boldsymbol{p}+\omega^2\nabla\cdot\sigma(\frac{\rho_0\boldsymbol{w}}{\rho_0-\rho_1})={\boldsymbol 0},~~~{\rm in}\ D,
\end{equation}
and the boundary conditions
\begin{equation}
{\boldsymbol w} = \sigma({\boldsymbol w}) {\boldsymbol \nu} ={\boldsymbol 0},~~~{\rm on}\ \partial D.
\end{equation}

Following the discussion in \cite{Ji2017}, we have $V\subset (H_0^2(D))^2$. That is, the boundary condition $\sigma({\boldsymbol w}) {\boldsymbol \nu} ={\boldsymbol 0}$ implies that all the first derivatives of ${\boldsymbol w}$ vanish on the boundary $\partial D$.  Further,  assume that the difference of mass density $\rho_0({\boldsymbol x})-\rho_1({\boldsymbol x})$ is smooth enough. For the fourth order equation \eqref{Fourth_regularity} with homogeneous boundary conditions, we can obtain $\boldsymbol{w}\in (H^3(D))^2\cap V$ (see, for instance \cite{Camano2018} and \cite{Girault1986}), which, together with \eqref{Problem}, implies that $\boldsymbol{p}\in (H^1(D))^2$.

Before introducing the weak variational formulation, we denote the inner product of two square matrices $A$ and $B$
\begin{equation}(A,B)=\int_D A:B \ {\rm d}{\boldsymbol x} = \int_D {\rm tr}(AB^\top) {\rm d}{\boldsymbol x},
\end{equation}
where $A:B={\rm tr}(AB^\top)$, i.e., the Frobenius inner product of $A$ and $B$.

Multiplying equations (\ref{Problem}) by suitable test functions and integrating by parts, the corresponding weak formulations are obtained: Find $\omega^2\in \mathbb{C}$ and non vanishing
$({\boldsymbol w},{\boldsymbol p})\in (H_0^1(D))^2\times(H^1(D))^2$ such that
\begin{equation}\label{EigenWeakFormulation}
\left\{
\begin{array}{rcll}
(\sigma({\boldsymbol w}),\nabla {\boldsymbol \phi})+((\rho_0-\rho_1){\boldsymbol p},{\boldsymbol \phi})&=&\omega^2(\rho_0{\boldsymbol w},{\boldsymbol \phi}),&\forall {\boldsymbol \phi}\in(H^1(D))^2,\\
(\sigma({\boldsymbol p}),\nabla {\boldsymbol \varphi})&=&\omega^2(\rho_1{\boldsymbol p},{\boldsymbol \varphi}),&\forall {\boldsymbol \varphi}\in(H_0^1(D))^2.
\end{array}
\right.
\end{equation}

Denote the bilinear form $a({\boldsymbol w},{\boldsymbol \phi})=\int_D \sigma({\boldsymbol w})\colon\nabla{\boldsymbol \phi}\,dx$. It's easy to verify that
\begin{equation}\label{sigmaugv}
a({\boldsymbol w},{\boldsymbol \phi})=\int_D 2\mu\varepsilon(\boldsymbol{w})\colon\varepsilon(\boldsymbol{\phi})+\lambda(\nabla\cdot\boldsymbol{w})(\nabla\cdot\boldsymbol{\phi})\,{\rm d}{\boldsymbol x},\ \ \ \forall \ {\boldsymbol w}, {\boldsymbol \phi}\in(H^1(D))^2.
\end{equation}
By the first Korn inequality \cite[Corollary 11.2.25]{BrennerScott2002}, there exists a positive constant $C$ such that
\[
\|\varepsilon({\boldsymbol u})\|_{L^2} \ge C \|{\boldsymbol
u}\|_{H^1},\quad \forall {\boldsymbol u} \in (H_0^1(D))^2,
\]
which naturally guarantees the coercivity of \eqref{sigmaugv} for $\lambda>0$.
Especially, if $\lambda<0$, we can derive that
\begin{eqnarray}
a({\boldsymbol u},{\boldsymbol u})&=&\int_D 2\mu (\partial_{x}^2u_1+\partial_{y}^2u_2)+\mu(\partial_{y}u_1+\partial_{x}u_2)^2+\lambda(\partial_{x}u_1+\partial_{y}u_2)^2{\rm d}{\boldsymbol x}\\ \nonumber
&\geq&\int_D 2(\lambda+\mu)(\partial_{x}^2u_1+\partial_{y}^2u_2)+\mu(\partial_{y}u_1+\partial_{x}u_2)^2{\rm d}{\boldsymbol x} \\ \nonumber
&\geq& 2(\lambda+\mu)\int_D\varepsilon(\boldsymbol{u})\colon\varepsilon(\boldsymbol{u}) {\rm d}{\boldsymbol x}.
\end{eqnarray}
which also implies the coercivity of \eqref{sigmaugv}.

The following lemma shows that the variational problem \eqref{EigenWeakFormulation} is equivalent to the original one \eqref{tep}.
\begin{lemma}
If $(\omega^2,{\boldsymbol u},{\boldsymbol v})$ is a solution of \eqref{tep}, then $(\omega^2,{\boldsymbol {u-v}},-\omega^2{\boldsymbol v})$ is a solution of \eqref{EigenWeakFormulation}.
On the other hand, if $(\omega^2,{\boldsymbol w},{\boldsymbol p})$ is the solution of \eqref{EigenWeakFormulation}, then $(\omega^2,{\boldsymbol w}-{\boldsymbol p}/\omega^2,-{\boldsymbol p}/\omega^2)$ is the solution of \eqref{tep}.
\end{lemma}
\begin{proof}
The first part is straightforward. For the converse, let $(\omega^2,{\boldsymbol w},{\boldsymbol p})$ be a solution of \eqref{EigenWeakFormulation}. Then, we conclude $\omega^2\neq0$. Otherwise, by reduction to absurdity, taking ${\boldsymbol \phi}={\boldsymbol p}$ and ${\boldsymbol \varphi}={\boldsymbol w}$, it follows that ${\boldsymbol p}={\boldsymbol 0}$ and further ${\boldsymbol w}={\boldsymbol 0}$, here we use the coercivity of the bilinear form $a(\cdot,\cdot)$. In \eqref{EigenWeakFormulation}, using integration by parts, we can obtain that ${\boldsymbol w},{\boldsymbol p}$ satisfy the equations of system \eqref{Problem} and the boundary condition $\sigma({\boldsymbol w}){\boldsymbol \nu}={\boldsymbol 0}$.
Since $D$ is convex, we have ${\boldsymbol w}\in V$. Thus, $(\omega^2,{\boldsymbol w},{\boldsymbol p})$ is a solution to the system \eqref{Problem}.
By the equivalence of \eqref{tep} and \eqref{Problem}, it's easy to check that $(\omega^2,{\boldsymbol w}-{\boldsymbol p}/\omega^2,-{\boldsymbol p}/\omega^2)$ is also the solution of \eqref{tep}.
\end{proof}

The corresponding source problem of \eqref{EigenWeakFormulation} is stated as follows:
Given $({\boldsymbol f},{\boldsymbol g})\in (H_0^1(D))^2\times (L^2(D))^2$, find $( {\boldsymbol w}, {\boldsymbol p})\in (H_0^1(D))^2\times (H^1(D))^2$ satisfying
\begin{equation}\label{SourceWeakFormulation}
\left\{
\begin{array}{rcll}
(\sigma({\boldsymbol w}),\nabla {\boldsymbol \phi})+((\rho_0-\rho_1){\boldsymbol p},{\boldsymbol \phi})&=&(\rho_0{\boldsymbol f},{\boldsymbol \phi}),&\forall {\boldsymbol \phi}\in(H^1(D))^2,\\
(\sigma({\boldsymbol p}),\nabla {\boldsymbol \varphi})&=&(\rho_1{\boldsymbol g},{\boldsymbol \varphi}),&\forall {\boldsymbol \varphi}\in(H_0^1(D))^2.
\end{array}
\right.
\end{equation}

To solve the new non-self-adjoint eigenvalue system \eqref{SourceWeakFormulation}, we define the sesquilinear forms $\mathcal A, \mathcal B$ on $((H_0^1(D))^2\times (H^1(D))^2)
\times ((H_0^1(D))^2\times (H^1(D))^2)$ as follows

\begin{eqnarray}
&\mathcal{A}\big(({\boldsymbol w},{\boldsymbol p}),({\boldsymbol \varphi},{\boldsymbol \phi})\big)=(\sigma({\boldsymbol w}),\nabla {\boldsymbol \phi})+((\rho_0-\rho_1){\boldsymbol p},{\boldsymbol \phi})+(\sigma({\boldsymbol p}),\nabla {\boldsymbol \varphi}), \nonumber\\
&\mathcal{B}\big(({\boldsymbol w},{\boldsymbol p}),({\boldsymbol \varphi},{\boldsymbol \phi})\big)=(\rho_0{\boldsymbol w},{\boldsymbol \phi})+(\rho_1{\boldsymbol p},{\boldsymbol \varphi}),\nonumber
\end{eqnarray}
here $\boldsymbol w, \boldsymbol \varphi \in (H_0^1(D))^2,  \boldsymbol p,\boldsymbol \phi \in (H^1(D))^2$. Note that $\mathcal A$ is a inner product on $((H_0^1(D))^2\times (H^1(D))^2) \times ((H_0^1(D))^2\times (H^1(D))^2)$

Then the eigenvalue problem \eqref{SourceWeakFormulation} can be formulated as: Find $\omega\in \mathbb C$ and non-trivial $({\boldsymbol w},{\boldsymbol p})\in ((H_0^1(D))^2\times (H^1(D))^2)$ such that
\begin{equation}\label{VariationallyPosedEVP}
\frac{1}{\omega^2}\mathcal{A}\big(({\boldsymbol w},{\boldsymbol p}),({\boldsymbol \varphi},{\boldsymbol \phi})\big)=\mathcal{B}\big(({\boldsymbol w},{\boldsymbol p}),({\boldsymbol \varphi},{\boldsymbol \phi})\big), \, \forall ({\boldsymbol \varphi},{\boldsymbol \phi})\in (H_0^1(D))^2\times (H^1(D))^2,
\end{equation}
note that $\omega=0$ is not a transmission eigenvalue.

Using \eqref{SourceWeakFormulation}, we can define the solution operator $T\colon (H_0^1(D))^2\times (L^2(D))^2 \rightarrow (H_0^1(D))^2\times (L^2(D))^2$ by
\begin{eqnarray}\label{TTTTTT}
 \mathcal{A}\big(T({\boldsymbol w},{\boldsymbol p}),({\boldsymbol \varphi},{\boldsymbol \phi})\big)=\mathcal{B}\big(({\boldsymbol w},{\boldsymbol p}),({\boldsymbol \varphi},{\boldsymbol \phi})\big), \, \forall ({\boldsymbol \varphi},{\boldsymbol \phi})\in (H_0^1(D))^2\times (H^1(D))^2.
\end{eqnarray}
It's equivalent to the definition
 $({\boldsymbol f},{\boldsymbol g}) \mapsto T({\boldsymbol f},{\boldsymbol g})=({\boldsymbol w},{\boldsymbol p})$. Here $({\boldsymbol w},{\boldsymbol p})$ is the solution of \eqref{SourceWeakFormulation}.

Then we seek $\omega\in\mathbb C$ and non-trivial $({\boldsymbol w},{\boldsymbol p})\in (H_0^1(D))^2\times (H^1(D))^2$ such that
\begin{equation}
T({\boldsymbol w},{\boldsymbol p})=\frac{1}{\omega^2}({\boldsymbol w},{\boldsymbol p}).
\end{equation}
No spurious eigenvalues are introduced into the system since if $\omega\neq 0$, $({\boldsymbol 0},{\boldsymbol p})$ is not an eigenfunction of this
system. The above discussion gives a consistent one-to-one match
between the eigenvalue system \eqref{EigenWeakFormulation} and the compact operator $T$. We write it into the following theorem.
\begin{theorem} \label{Tcontinuous}
If $(\lambda,{\boldsymbol w},{\boldsymbol p})$ is an eigenpair of $T$ with $\lambda\neq0$, then $(\omega^2,{\boldsymbol w},{\boldsymbol p})$ is the solution of \eqref{EigenWeakFormulation} with $\lambda=\frac{1}{\omega^2}$, and vice versa.
\end{theorem}

The following Lemma shows the well-posedness of $T$.
\begin{lemma}\label{regulazing}
Given $({\boldsymbol f}, {\boldsymbol g})\in (H_0^1(D))^2\times (L^2(D))^2$, the boundary-value problem \eqref{SourceWeakFormulation} has a unique solution $({\boldsymbol w}, {\boldsymbol p})\in (H_0^1(D))^2\times (H^1(D))^2$. Typically we have
$\boldsymbol{w}\in (H^3(D))^2\cap V$ provided that $\rho_0({\boldsymbol x})-\rho_1({\boldsymbol x})$ is smooth enough.
\end{lemma}
\begin{proof}
First, we prove the uniqueness of the solution. Assume $({\boldsymbol f}, {\boldsymbol g})=({\boldsymbol 0}, {\boldsymbol 0})$, from \eqref{SourceWeakFormulation}, we have
\begin{equation}\label{ppppppppppp}
\left\{
\begin{array}{rcll}
(\sigma({\boldsymbol w}),\nabla {\boldsymbol \phi})+((\rho_0-\rho_1){\boldsymbol p},{\boldsymbol \phi})&=&0,&\forall {\boldsymbol \phi}\in(H^1(D))^2,\\
(\sigma({\boldsymbol p}),\nabla {\boldsymbol \varphi})&=&0,&\forall {\boldsymbol \varphi}\in(H_0^1(D))^2.
\end{array}
\right.
\end{equation}
Taking ${\boldsymbol \phi}={\boldsymbol p},\ {\boldsymbol \varphi}={\boldsymbol w}$ and combining \eqref{sigmaugv}, we can obtain
\begin{equation}
(\sigma({\boldsymbol w}),\nabla {\boldsymbol p})=(\sigma({\boldsymbol p}),\nabla {\boldsymbol w})=0.
\end{equation}
Using the first equation of \eqref{ppppppppppp}, it follows that ${\boldsymbol p}={\boldsymbol 0}$. Then by taking ${\boldsymbol \phi}={\boldsymbol w}$, we conclude that ${\boldsymbol w}={\boldsymbol 0}$. Hence, problem \eqref{SourceWeakFormulation} admits at most one solution.

Next, we prove the existence of the solution. For $({\boldsymbol f}, {\boldsymbol g})\in (H_0^1(D))^2\times (L^2(D))^2$, let ${\boldsymbol w}\in V$ satisfying
\begin{equation} \label{auxiliaryequation}
\int_{D}\left(\frac{\nabla\cdot\sigma({\boldsymbol w})}{\rho_0-\rho_1}\right)\cdot\left(\nabla\cdot\sigma({\boldsymbol \phi})\right)\,dx = -\int_{D}\rho_1{\boldsymbol g}\cdot{\boldsymbol \phi}\,dx
+\int_{D}\sigma\left(\frac{\rho_0\boldsymbol f}{\rho_0-\rho_1}\right)\colon\nabla{\boldsymbol \phi}\,dx,\\\\\ \forall {\boldsymbol \phi}\in V.
\end{equation}

Using the results in \cite{Ji2017, MarsdenHughes1994}, we know
\begin{equation}
\Vert \nabla\cdot\sigma({\boldsymbol \psi})\Vert_{0,D} \geq c\Vert{\boldsymbol \psi}\Vert_{2,D}, \ \ \ \forall {\boldsymbol \psi}\in V,  \nonumber
\end{equation}
which guarantees the coercivity of  equation \eqref{auxiliaryequation}.
As a consequence of Lax-Milgram theorem, there exists a unique ${\boldsymbol w}$ satisfying \eqref{auxiliaryequation}. Taking the integration by parts in \eqref{auxiliaryequation}, we can obtain ${\boldsymbol w}\in V$ satisfying
\begin{equation}\label{A}
-\nabla\cdot\sigma\left(\frac{\nabla\cdot\sigma({\boldsymbol w})}{\rho_0-\rho_1}\right)=\rho_1{\boldsymbol g}+\nabla\cdot\sigma\left(\frac{\rho_0{\boldsymbol f}}{\rho_0-\rho_1}\right)\in H^{-1}(D),
\end{equation}
which implies that $\boldsymbol{w}\in (H^3(D))^2\cap V$ for smooth $\rho_0({\boldsymbol x})-\rho_1({\boldsymbol x})$ \cite{Camano2018,Girault1986}.

Further, define ${\boldsymbol p}$ as
\begin{equation}\label{B}
{\boldsymbol p}=\frac{\rho_0{\boldsymbol f}+\nabla\cdot\sigma({\boldsymbol w})}{\rho_0-\rho_1}.
\end{equation}
Combining \eqref{A} and \eqref{B}, we have
\begin{equation}
-\nabla\cdot\sigma({\boldsymbol p})=-\nabla\cdot\sigma\left(\frac{\rho_0{\boldsymbol f}+\nabla\cdot\sigma({\boldsymbol w})}{\rho_0-\rho_1}\right)=\rho_1{\boldsymbol g},
\end{equation}
Thus, $({\boldsymbol w},{\boldsymbol p})$ is a solution of \eqref{SourceWeakFormulation}. The proof is complete.
 \end{proof}

\begin{lemma}
$T$ is a linear, bounded and compact operator.
\end{lemma}
\begin{proof}
It's easy to verify the linearity and boundedness of $T$. Here we only need to prove the compactness.
It's a consequence of the fact that
\begin{equation}
T\left((H_0^1(D))^2\times (L^2(D))^2\right)\subset\left[(H^3(D))^2\cap V\right]\times (H^1(D))^2\hookrightarrow(H_0^1(D))^2\times (L^2(D))^2.
\end{equation}
The second inclusion is compact.
\end{proof}

\section{Error estimates of the eigenpair approximation}\label{ErrorEstimate}
In this section, we consider the Galerkin finite element method for the elasticity transmission eigenvalue problem. First, some notations are introduced.
Let $\{\mathcal{T}_h\}_{h>0}$ be a family of shape regular meshes over $D$ with mesh size $h$, $V_h\subset (H^1(D))^2$ be $(P_1)^2$ Lagrange finite element space associated with
$\mathcal{T}_h$ and $V_h^0=V_h\cap(H_0^1(D))^2$. A lowest-order finite element method is studied here. We follow the approach from \cite{Camano2018,Ciarlet1974}.

The Galerkin approximation for problem \eqref{EigenWeakFormulation} is:
Find $\omega_h^2\in \mathbb{C}$ and
$({\boldsymbol w_h},{\boldsymbol p_h})\in V_h^0\times V_h$ such that
\begin{equation}\label{EigenWeakDiscreteFormulation}
\left\{
\begin{array}{rcll}
(\sigma({\boldsymbol w_h}),\nabla {\boldsymbol \phi_h})+((\rho_0-\rho_1){\boldsymbol p_h},{\boldsymbol \phi_h})&=&\omega_h^2(\rho_0{\boldsymbol w_h},{\boldsymbol \phi_h}),&\forall {\boldsymbol \phi_h}\in V_h,\\
(\sigma({\boldsymbol p_h}),\nabla {\boldsymbol \varphi_h})&=&\omega_h^2(\rho_1{\boldsymbol p_h},{\boldsymbol \varphi_h}),&\forall {\boldsymbol \varphi_h}\in V_h^0.
\end{array}
\right.
\end{equation}
The system \eqref{EigenWeakDiscreteFormulation}  is equivalent to the following formulation: Find $\omega_h^2\in \mathbb{C}$ and
$({\boldsymbol w_h},{\boldsymbol p_h})\in V_h^0\times V_h$ such that
 \begin{equation}\label{VariationallyPosedEVP}
\mathcal{A}\big(({\boldsymbol w_h},{\boldsymbol p_h}),({\boldsymbol \varphi_h},{\boldsymbol \phi_h})\big)=\omega_h^2\mathcal{B}\big(({\boldsymbol w_h},{\boldsymbol p_h}),({\boldsymbol \varphi_h},{\boldsymbol \phi_h})\big), \, \forall ({\boldsymbol \varphi_h},{\boldsymbol \phi_h})\in V_h^0\times V_h.
\end{equation}
If we want to employ the approximation theory of variationally posed eigenvalue problems \cite{BabuskaOsborn1991,Kolata1978,MercierOsbornRappazRaviart1981},
the following two conditions are required
\begin{equation}
\sup\limits_{({\boldsymbol w_h},{\boldsymbol p_h})\in V_h^0\times V_h} \frac{\mathcal{A}\big(({\boldsymbol w_h},{\boldsymbol p_h}),({\boldsymbol \varphi_h},{\boldsymbol \phi_h})\big)}{\Vert{\boldsymbol w_h}\Vert_{1,D}+\Vert{\boldsymbol p_h}\Vert_{1,D}}\geq \beta(h)\big(\Vert{\boldsymbol \varphi_h}\Vert_{1,D}+\Vert{\boldsymbol \phi_h}\Vert_{1,D} \big)>0
\end{equation}
and
\begin{equation}
\lim\limits_{h\rightarrow 0}\Big[\frac{1}{\beta(h)} \inf_{({\boldsymbol \varphi_h},{\boldsymbol \phi_h})\in V_h^0\times V_h}\big( \Vert{\boldsymbol w}-{\boldsymbol \varphi_h}\Vert_{1,D}+\Vert {\boldsymbol p}-{\boldsymbol \phi_h}\Vert_{1,D} \big) \Big]=0, \ \ \ \ \forall ({\boldsymbol w},{\boldsymbol p})\in (H_0^1(D))^2\times (H^1(D))^2,
\end{equation}
however, to the best of our knowledge, which are unsatisfied due to the non-self-adjointness. So we can not use the classical theoretical analysis  for mixed eigenvalue problems \eqref{EigenWeakFormulation} and \eqref{EigenWeakDiscreteFormulation}. Instead, we resort to the spectral approximation of compact operator \cite{BabuskaOsborn1991} and try to
prove the convergence of the approximate solution operator $T_h$ to $T$.

We introduce the approximate source problem: Given $({\boldsymbol f},{\boldsymbol g})\in (H_0^1(D))^2\times (L^2(D))^2$, find $({\boldsymbol w_h}, {\boldsymbol p_h})\in V_h^0\times V_h$ satisfying
\begin{equation}\label{SourceWeakDiscreteFormulation}
\left\{
\begin{array}{rcll}
(\sigma({\boldsymbol w_h}),\nabla {\boldsymbol \phi_h})+((\rho_0-\rho_1){\boldsymbol p_h},{\boldsymbol \phi_h})&=&(\rho_0{\boldsymbol f},{\boldsymbol \phi_h}),&\forall {\boldsymbol \phi_h}\in V_h,\\
(\sigma({\boldsymbol p_h}),\nabla {\boldsymbol \varphi_h})&=&(\rho_1{\boldsymbol g},{\boldsymbol \varphi_h}),&\forall {\boldsymbol \varphi_h}\in V_h^0.
\end{array}
\right.
\end{equation}
Then, the corresponding discrete solution operator is defined as follows
\begin{eqnarray}\label{Thhhhhh}
T_h\colon (H_0^1(D))^2\times (L^2(D))^2 &\rightarrow& (H_0^1(D))^2\times (L^2(D))^2,  \\ \nonumber
({\boldsymbol f},{\boldsymbol g}) &\mapsto& T_h({\boldsymbol f},{\boldsymbol g})=({\boldsymbol w_h},{\boldsymbol p_h}),
\end{eqnarray}
where $({\boldsymbol w_h},{\boldsymbol p_h})$ is the numerical solution of \eqref{SourceWeakDiscreteFormulation}.

\begin{lemma}
$T_h$ is well-defined. What's more, $T_h$ a bounded, bilinear, compact operator.
\end{lemma}
 \begin{proof}
The proof follows the related results for $T$.
\end{proof}

The next theorem gives a consistent one-to-one match
between the eigenvalue system \eqref{EigenWeakDiscreteFormulation} and the compact operator $T_h$.
\begin{theorem}
If $(\lambda_h,{\boldsymbol w_h},{\boldsymbol p_h})$ is an eigenpair of $T_h$ with $\lambda_h\neq0$, then $(\omega_h^2,{\boldsymbol w_h},{\boldsymbol p_h})$ is the solution of \eqref{EigenWeakDiscreteFormulation} with $\lambda_h=\frac{1}{\omega_h^2}$, and vice versa.
\end{theorem}
 \begin{proof}
The result is a discrete version of  Theorem \ref{Tcontinuous}.
  \end{proof}

The idea of the section comes from the analysis of the stream function-vorticity-pressure formulation
of the Stokes problem. The following analysis is similar to Theorem III.2.6 and Lemma III.3.1 from \cite{Girault1986}.
First, we define the projection operator $P_h\colon (H^1(D))^2\rightarrow V_h$ satisfying
\begin{equation}
\left\{
\begin{array}{rlcl}
&a({\boldsymbol p}-P_h{\boldsymbol p},{\boldsymbol \theta_h}) &=& 0,~~~\forall  {\boldsymbol \theta_h}\in V_h, \\ \nonumber
&({\boldsymbol p}-P_h{\boldsymbol p},1) &=& {\boldsymbol 0}.
\end{array}
 \right.
\end{equation}
We also introduce the following two sets:
\begin{equation}
V({\boldsymbol f})=\{({\boldsymbol \psi},{\boldsymbol \varphi})\in (H_0^1(D))^2\times(L^2(D))^2\colon(\sigma({\boldsymbol \psi}),\nabla {\boldsymbol z})+((\rho_0-\rho_1){\boldsymbol \varphi},{\boldsymbol z})=(\rho_0{\boldsymbol f},{\boldsymbol z}),~\forall {\boldsymbol z}\in (H^1(D))^2\}, \ \nonumber
\end{equation}
and
\begin{equation}
V_h({\boldsymbol f})=\{({\boldsymbol \psi_h},{\boldsymbol \varphi_h})\in V_h^0\times V_h\colon(\sigma({\boldsymbol \psi_h}),\nabla {\boldsymbol z_h})+((\rho_0-\rho_1){\boldsymbol \varphi_h},{\boldsymbol z_h})=(\rho_0{\boldsymbol f},{\boldsymbol z_h}),~\forall {\boldsymbol z_h}\in V_h\}.  \ \nonumber
\end{equation}

First, we give the following auxiliary lemma.
\begin{lemma}\label{lemma3333}
Given $({\boldsymbol f},{\boldsymbol g})\in (H_0^1(D))^2\times (L^2(D))^2$, let $({\boldsymbol w}, {\boldsymbol p})\in (H_0^1(D))^2\times (H^1(D))^2$ be the solution of \eqref{SourceWeakFormulation} and $({\boldsymbol w_h}, {\boldsymbol p_h})\in V_h^0\times V_h$ be the numerical solution of \eqref{SourceWeakDiscreteFormulation}. Then, the following error estimate holds
\begin{equation}
\parallel {\boldsymbol w}-{\boldsymbol w_h}\parallel_{1,D}+\parallel {\boldsymbol p}-{\boldsymbol p_h}\parallel_{0,D}\leq \inf\limits_{({\boldsymbol \psi_h},{\boldsymbol \varphi_h})\in V_h ({\boldsymbol f})}(\parallel {\boldsymbol w}-{\boldsymbol \psi_h}\parallel_{1,D}+\parallel {\boldsymbol p}-{\boldsymbol \varphi_h}\parallel_{0,D})+
C\parallel {\boldsymbol p}-P_h{\boldsymbol p}\parallel_{0,D},
\end{equation}
where $C>0$ is a constant independent of $h, {\boldsymbol f}, {\boldsymbol g}$.
\end{lemma}
\begin{proof}
For $\forall({\boldsymbol \psi_h},{\boldsymbol \varphi_h})\in V_h ({\boldsymbol f})$, using triangle inequality, we can obtain
\begin{equation}\label{000000}
\parallel {\boldsymbol w}-{\boldsymbol w_h}\parallel_{1,D}+\parallel {\boldsymbol p}-{\boldsymbol p_h}\parallel_{0,D}\leq \parallel {\boldsymbol w}-{\boldsymbol \psi_h}\parallel_{1,D}+\parallel{\boldsymbol \psi_h}-{\boldsymbol w_h}\parallel_{1,D}+\parallel {\boldsymbol p}-{\boldsymbol \varphi_h}\parallel_{0,D}+\parallel{\boldsymbol\varphi_h}-{\boldsymbol p_h}\parallel_{0,D}.
\end{equation}
Next, we bound the terms $\parallel{\boldsymbol \psi_h}-{\boldsymbol w_h}\parallel_{1,D}$ and $\parallel{\boldsymbol\varphi_h}-{\boldsymbol p_h}\parallel_{0,D}$,  respectively.
For the term $\parallel{\boldsymbol\varphi_h}-{\boldsymbol p_h}\parallel_{0,D}$, we have
\begin{eqnarray} \label{aaaaaa}\nonumber
\rho_{min}\Vert{\boldsymbol \varphi_h}-{\boldsymbol p_h}\Vert^2_{0,\Omega}&\leq&
\big\vert\big((\rho_0-\rho_1)({\boldsymbol \varphi_h}-{\boldsymbol p_h}),({\boldsymbol \varphi_h}-{\boldsymbol p_h})\big)\big\vert \\ 
&=& \big\vert\big((\rho_0-\rho_1)({\boldsymbol \varphi_h}-{\boldsymbol p_h}),({\boldsymbol \varphi_h}-P_h{\boldsymbol p})\big) + \big((\rho_0-\rho_1)({\boldsymbol \varphi_h}-{\boldsymbol p_h}),(P_h{\boldsymbol p}-{\boldsymbol p_h})\big)\big \vert.
\end{eqnarray}
Since $({\boldsymbol w_h}, {\boldsymbol p_h}), ({\boldsymbol \psi_h}, {\boldsymbol \varphi_h})\in V_h({\boldsymbol f})$, we can obtain
\begin{equation}\label{bbbbbb}
\big(\sigma({\boldsymbol w_h}-{\boldsymbol \psi_h}),\nabla{\boldsymbol z_h}\big) +  \big((\rho_0-\rho_1)({\boldsymbol p_h}-{\boldsymbol \varphi_h}),{\boldsymbol z_h}\big)=0,~~~~\forall {\boldsymbol z_h}\in V_h.
\end{equation}
In particular, taking ${\boldsymbol z_h}=P_h{\boldsymbol p}-{\boldsymbol p_h}$, we have
\begin{equation}\label{pppp}
\big((\rho_0-\rho_1)({\boldsymbol p_h}-{\boldsymbol \varphi_h}),P_h{\boldsymbol p}-{\boldsymbol p_h}\big)=-\big(\sigma({\boldsymbol w_h}-{\boldsymbol \psi_h}),\nabla(P_h{\boldsymbol p}-{\boldsymbol p_h})\big).
\end{equation}
On the other hand, the combination of \eqref{SourceWeakFormulation} and \eqref{SourceWeakDiscreteFormulation} leads to
\begin{equation}
\big(\sigma({\boldsymbol p}),\nabla{\boldsymbol \theta_h} \big)=\big(\sigma({\boldsymbol p_h}),\nabla{\boldsymbol \theta_h} \big),~~~~\forall  {\boldsymbol \theta_h}\in V_h^0.
\end{equation}
Further, the following result holds
\begin{equation}
\big(\sigma({\boldsymbol p_h}-P_h{\boldsymbol p}),\nabla{\boldsymbol \theta_h} \big)=\big(\sigma({\boldsymbol p}-P_h{\boldsymbol p}),\nabla{\boldsymbol \theta_h} \big)=0,~~~~\forall  {\boldsymbol \theta_h}\in V_h^0.
\end{equation}
Especially, taking ${\boldsymbol \theta_h}={\boldsymbol \psi_h}-{\boldsymbol w_h}$ and combining \eqref{pppp}, we have
\begin{equation}
\big((\rho_0-\rho_1)({\boldsymbol p_h}-{\boldsymbol \varphi_h}),P_h{\boldsymbol p}-{\boldsymbol p_h}\big)=0.
\end{equation}
Using the above  equation \eqref{aaaaaa} and the Cauchy-Schwarz inequality, we obtain
\begin{eqnarray} \nonumber
\rho_{min}\Vert{\boldsymbol \varphi_h}-{\boldsymbol p_h}\Vert^2_{0,D}&\leq&
\big\vert\big((\rho_0-\rho_1)({\boldsymbol \varphi_h}-{\boldsymbol p_h}),({\boldsymbol \varphi_h}-{\boldsymbol p_h})\big)\big\vert \\ \nonumber
&=& \big\vert\big((\rho_0-\rho_1)({\boldsymbol \varphi_h}-{\boldsymbol p_h}),({\boldsymbol \varphi_h}-P_h{\boldsymbol p})\big)\big \vert \\  
&\leq& \rho_{max}\Vert{\boldsymbol \varphi_h}-{\boldsymbol p_h}\Vert_{0,D} \Vert {\boldsymbol \varphi_h}-P_h{\boldsymbol p}\Vert_{0,D}.
\end{eqnarray}
Therefore, we have
\begin{equation}\label{111111}
\Vert{\boldsymbol \varphi_h}-{\boldsymbol p_h}\Vert_{0,D}\leq C \Vert {\boldsymbol \varphi_h}-P_h{\boldsymbol p}\Vert_{0,D}\leq C\Vert {\boldsymbol \varphi_h}-{\boldsymbol p}\Vert_{0,D}
+C\Vert {\boldsymbol p}-P_h{\boldsymbol p}\Vert_{0,D}.
\end{equation}
In the following, we consider the bound of the term $\parallel{\boldsymbol \psi_h}-{\boldsymbol w_h}\parallel_{1,D}$. In \eqref{bbbbbb}, taking ${\boldsymbol z_h}={\boldsymbol w_h}-{\boldsymbol \psi_h}$ and applying
the coercivity of \eqref{sigmaugv}, we obtain
\begin{equation}\label{cccccccc}
C_1\Vert{\boldsymbol w_h}-{\boldsymbol \psi_h}\Vert_{1,D}^2
\leq \rho_{max}\Vert {\boldsymbol p_h}-{\boldsymbol \varphi_h}\Vert_{0,D}
\Vert {\boldsymbol w_h}-{\boldsymbol \psi_h}\Vert_{0,D},
\end{equation}
where $C_1$ is a positive constant.
Futher, the following result holds
\begin{equation}
\Vert{\boldsymbol w_h}-{\boldsymbol \psi_h}\Vert_{1,D}^2\leq\frac{\rho_{max}}{C_1}\Vert {\boldsymbol p_h}-{\boldsymbol \varphi_h}\Vert_{0,D}
\Vert {\boldsymbol w_h}-{\boldsymbol \psi_h}\Vert_{1,D}.
\end{equation}
Then, we obtain
\begin{equation}\label{222222}
\Vert{\boldsymbol w_h}-{\boldsymbol \psi_h}\Vert_{1,D}\leq\frac{\rho_{max}}{C_1}\Vert {\boldsymbol p_h}-{\boldsymbol \varphi_h}\Vert_{0,D}.
\end{equation}
The combinations of \eqref{000000}, \eqref{111111} and \eqref{222222} complete the proof.
\end{proof}
Next, we give the estimates on each of the two terms on the right hand side of \eqref{000000}.
\begin{lemma}\label{lemma4444}
For $\forall({\boldsymbol \psi},{\boldsymbol \varphi})\in V({\boldsymbol f})$, there exists a constant $C>0$ such that
\begin{eqnarray}\label{ffffff}
\inf\limits_{({\boldsymbol \psi_h},{\boldsymbol \varphi_h})\in V_h({\boldsymbol f})}\big(\Vert{\boldsymbol \psi}-{\boldsymbol \psi_h}\Vert_{1,D}+\Vert{\boldsymbol \varphi}-{\boldsymbol \varphi_h}\Vert_{0,D}\big)
&\leq &C\inf\limits_{({\boldsymbol \theta_h},{\boldsymbol z_h})\in V_h^0\times V_h}
\Big[\big(\Vert{\boldsymbol \psi}-{\boldsymbol \theta_h}\Vert_{1,D}+\Vert{\boldsymbol \varphi}-{\boldsymbol z_h}\Vert_{0,D}\big) \nonumber\\
&+&\sup\limits_{{\boldsymbol y_h}\neq0\in V_h}\frac{\vert a({\boldsymbol \psi}-{\boldsymbol \theta_h},{\boldsymbol y_h})\vert}{\Vert {\boldsymbol y_h}\Vert_{0,D}}\Big].  
\end{eqnarray}
\end{lemma}
\begin{proof}
For $\forall$ $({\boldsymbol \psi},{\boldsymbol \varphi})\in V({\boldsymbol f})$, we have
\begin{equation}
(\sigma({\boldsymbol \psi}),\nabla {\boldsymbol y_h})+\big((\rho_0-\rho_1){\boldsymbol \varphi},{\boldsymbol y_h}\big)
=(\rho_0{\boldsymbol f},{\boldsymbol y_h}),~~~~\forall {\boldsymbol y_h}\in V_h. \ \nonumber
\end{equation}
For each $({\boldsymbol \theta_h},{\boldsymbol z_h})\in V_h^0\times V_h$,
define ${\boldsymbol \gamma_h}\in V_h$ by
\begin{equation}\label{dddddd}
 \big((\rho_0-\rho_1){\boldsymbol \gamma_h},{\boldsymbol y_h}\big)=\big(\sigma({\boldsymbol \psi}-{\boldsymbol \theta_h}),\nabla{\boldsymbol y_h}\big) +  \big((\rho_0-\rho_1)({\boldsymbol \varphi}-{\boldsymbol z_h}),{\boldsymbol y_h}\big),~~~~\forall {\boldsymbol y_h}\in V_h.
\end{equation}
Then, using \eqref{dddddd} and the triangle inequality, we obtain
\begin{eqnarray}\label{eeeeee}
\Vert{\boldsymbol \gamma_h}\Vert_{0,D}
&=&\sup\limits_{\Vert{\boldsymbol y_h}\Vert_{0,D}\neq0}\frac{\vert({\boldsymbol \gamma_h},{\boldsymbol y_h})\vert}{\Vert{\boldsymbol y_h}\Vert_{0,D}}
\leq \sup\limits_{\Vert{\boldsymbol y_h}\Vert_{0,D}\neq0}\frac{\vert\big((\rho_0-\rho_1){\boldsymbol \gamma_h},{\boldsymbol y_h}\big)\vert}{\rho_{min}\Vert{\boldsymbol y_h}\Vert_{0,D}} \nonumber\\
&\leq& \frac{\rho_{max}}{\rho_{min}}\Vert{\boldsymbol \varphi}-{\boldsymbol z_h}\Vert_{0,D}
+\sup\limits_{\Vert{\boldsymbol y_h}\Vert_{0,D}\neq0}\frac{\big\vert\big(\sigma({\boldsymbol \psi}-{\boldsymbol \theta_h}),\nabla{\boldsymbol y_h}\big) \big\vert}{\rho_{min}\Vert {\boldsymbol y_h}\Vert_{0,D}}.
\end{eqnarray}
From \eqref{dddddd}, it's easy to verify that
\begin{equation}
(\sigma({\boldsymbol \theta_h}),\nabla {\boldsymbol y_h})+((\rho_0-\rho_1)({\boldsymbol {z_h+\gamma_h}}),{\boldsymbol y_h})=(\rho_0{\boldsymbol f},{\boldsymbol y_h}),~\forall {\boldsymbol y_h}\in V_h.\nonumber
\end{equation}
Hence, if we define $({\boldsymbol \psi_h},{\boldsymbol \varphi_h}):=({\boldsymbol \theta_h},{\boldsymbol {z_h+\gamma_h}})$, then $({\boldsymbol \psi_h},{\boldsymbol \varphi_h})\in V_h({\boldsymbol f})$.
Using the triangle inequality, we have the following result
\begin{equation}
\parallel {\boldsymbol \psi}-{\boldsymbol \psi_h}\parallel_{1,D}+\parallel {\boldsymbol \varphi}-{\boldsymbol \varphi_h}\parallel_{0,D}\leq \parallel {\boldsymbol \psi}-{\boldsymbol \theta_h}\parallel_{1,D}+\parallel{\boldsymbol \varphi}-{\boldsymbol z_h}\parallel_{0,D}+\parallel{\boldsymbol\gamma_h}\parallel_{0,D}. \nonumber
\end{equation}
Combing the above equation and \eqref{eeeeee},  we get the proof.
\end{proof}
In the following, we estimate the last term on the right hand side of \eqref{ffffff}. We introduce the projection operator $P_{0,h}\colon(H_0^1(D))^2\rightarrow V^0_h$ defined by
\begin{equation}
a({\boldsymbol w}-P_{0,h}{\boldsymbol w},{\boldsymbol z_h}) = 0,~~~\forall  {\boldsymbol z_h}\in V^0_h \\ \nonumber
\end{equation}
and have the following result.
\begin{lemma}\label{lemma5555}
For $\forall {\boldsymbol \psi}\in (H^3(D))^2\cap(H_0^1(D))^2$ and each $\epsilon\in(0,\frac{1}{2})$, there exists  a constant $C(\epsilon)>0$, such that
\begin{equation}
\sup\limits_{{\boldsymbol y_h}\neq0\in V_h}\frac{\vert a({\boldsymbol \psi}-P_{0,h}{\boldsymbol \psi},{\boldsymbol y_h})\vert}{\Vert {\boldsymbol y_h}\Vert_{0,D}}\leq C(\epsilon) h^{\frac{1}{2}-\epsilon}\Vert {\boldsymbol \psi}\Vert_{3,D}.
\end{equation}
\end{lemma}
\begin{proof}
Analogous to Lemma \uppercase\expandafter{\romannumeral3}.3.2 of \cite{Girault1986} and combining the standard error estimate of the projection operator $P_{0,h}$, we can prove that, for a real number $p\geq 2$,
\begin{equation}
\sup\limits_{{\boldsymbol y_h}\neq0\in V_h}\frac{\vert a({\boldsymbol \psi}-P_{0,h}{\boldsymbol \psi},{\boldsymbol y_h})\vert}{\Vert {\boldsymbol y_h}\Vert_{0,D}}\leq C h^{\frac{1}{2}-\frac{1}{p}}\Vert {\boldsymbol \psi}\Vert_{2,p,D}.
\end{equation}
The Sobolev's embedding theorem (c.f. Theorem \uppercase\expandafter{\romannumeral1}.1.3 of \cite{Girault1986} ) implies that $H^3(D)\hookrightarrow H^{2,p}(D)$. Further, it's easy to verify that $(H^3(D))^2\hookrightarrow (H^{2,p}(D))^2$. Then, it yields that $\Vert {\boldsymbol \psi}\Vert_{2,p,D}\leq C(p)\Vert {\boldsymbol \psi}\Vert_{3,D}$.
Taking $p=\frac{1}{\epsilon}$, the proof is complete.
\end{proof}
The above three lemmas lead to the following result.
\begin{lemma}\label{ppppppppp}
Given $({\boldsymbol f},{\boldsymbol g})\in (H_0^1(D))^2\times (L^2(D))^2$, let $({\boldsymbol w}, {\boldsymbol p})\in (H_0^1(D))^2\times (H^1(D))^2$ be the solution of \eqref{SourceWeakFormulation} and $({\boldsymbol w_h}, {\boldsymbol p_h})\in V_h^0\times V_h$ be the numerical solution of \eqref{SourceWeakDiscreteFormulation}. Then, the following error estimate holds
\begin{equation}
\parallel {\boldsymbol w}-{\boldsymbol w_h}\parallel_{1,D}+\parallel {\boldsymbol p}-{\boldsymbol p_h}\parallel_{0,D}\leq
C\parallel {\boldsymbol w}-P_{0,h}{\boldsymbol w}\parallel_{1,D}
+ C\parallel {\boldsymbol p}-P_h{\boldsymbol p}\parallel_{0,D}
+C(\epsilon) h^{\frac{1}{2}-\epsilon}\Vert {\boldsymbol \psi}\Vert_{3,D}.\\ \nonumber
\end{equation}
\end{lemma}
\begin{proof}
Typically, we choose ${\boldsymbol \theta_h}=P_{0,h}{\boldsymbol w}$ and ${\boldsymbol z_h}=P_{h}{\boldsymbol p}$ in \eqref{ffffff}, then lemmas \ref{lemma3333}, \ref{lemma4444} and \ref{lemma5555} yield the result.
\end{proof}
\begin{lemma} \label{operatorcon}
Let $T_h\ (h>0)$ be a family of operators defined by \eqref{Thhhhhh} and $T$ defined by \eqref{TTTTTT}. Then, it follows that $\lim\limits_{h\rightarrow 0}T_h=T$, i.e. $\Vert T-T_h\Vert\rightarrow 0\  (h\rightarrow 0)$.
\end{lemma}
\begin{proof}
Given $({\boldsymbol f},{\boldsymbol g})\in (H_0^1(D))^2\times (L^2(D))^2$, let $({\boldsymbol w},{\boldsymbol p})=T({\boldsymbol f},{\boldsymbol g})$ and $({\boldsymbol w_h},{\boldsymbol p_h})=T_h({\boldsymbol f},{\boldsymbol g})$. Due to the definition of operator norm, we have
\begin{equation}
\Vert T-T_h\Vert
=\sup\limits_{({\boldsymbol f},{\boldsymbol g})\in (H_0^1(D))^2\times (L^2(D))^2}\frac{\Vert(T-T_h)({\boldsymbol f},{\boldsymbol g})\Vert}{\Vert({\boldsymbol f},{\boldsymbol g})\Vert}
=\sup\limits_{({\boldsymbol f},{\boldsymbol g})\in (H_0^1(D))^2\times (L^2(D))^2}\frac{\Vert({\boldsymbol w}-{\boldsymbol w_h},{\boldsymbol p}-{\boldsymbol p_h})\Vert}{\Vert({\boldsymbol f},{\boldsymbol g})\Vert}
\end{equation}
Using Lemma \ref{ppppppppp}, the proof is complete with the help of the standard error estimates for $P_h$ and $P_{0,h}$ \cite{Girault1986}.
\end{proof}

Let $\lambda$ be a nonzero eigenvalue of $T$ with algebraic multiplicity $m$, i.e. $\lambda\in\sigma(T)$. Lemma \ref{operatorcon}  tells us that for $h$ sufficiently small, there exist exactly $m$ eigenvalues $\lambda_h^k\ (k=1,\cdots,m)$ of $T_h$ such that $\lambda^k_h\rightarrow\lambda \ (h\rightarrow0)$.  Define the direct sum of the spaces of generalized eigenvectors corresponding to $\lambda_h^k(k=1,\cdots,m)$ as $W_h$, $\lambda_h(k=1,\cdots,m)$ as $W$.
The spectral theory for compact operators \cite{BabuskaOsborn1991,Osborn1975MC} gives the following theorem.
\begin{theorem}
Denote $\hat{\delta}(W,W_h)$ the gap between $W$ and $W_h$ by
\begin{equation}
\hat{\delta}(W,W_h)=\max(\delta(W,W_h),\delta(W_h,W))\ \ \ \  \ \nonumber
\end{equation}
here
\begin{equation}
\delta(W,W_h)=\sup\limits_{ \substack{({\boldsymbol \psi},{\boldsymbol \varphi})\in W,\\  \Vert{\boldsymbol \psi}\Vert_{1,D}^2+\Vert{\boldsymbol \varphi}\vert_{0,D}^2=1}}\Big[\inf\limits_{({\boldsymbol \psi_h},{\boldsymbol \varphi_h})\in W_h}\big(\Vert{\boldsymbol \psi}-{\boldsymbol \psi_h}\Vert_{1,D}^2+\Vert{\boldsymbol \varphi}-{\boldsymbol \varphi_h}\Vert_{0,D}^2\big)^{\frac{1}{2}}\Big],  \nonumber
\end{equation}
and $\delta(W_h,W)$ follows similarly. Then, $\hat{\delta}(W,W_h)\rightarrow 0$ as $h\rightarrow 0$.
\end{theorem}

\section{Numerical Examples}\label{Experiments}

In this section, we present some numerical results using three domains:
a disk with radius $R = 1/2$, the unit square and an L-shaped domain given by $(0,1)\times (0,1) \setminus [1/2, 1]\times[1/2,1]$.
Five levels of uniformly refined triangular meshes are generated for numerical experiments. The mesh size of initial mesh is $h_1=0.1$
and $h_i=h_{i-1}/2, i=2,3,4,5$.
Note that further refinement would lead to very large matrix eigenvalue problems which take too long to solve.
All examples are done using Matlab 2016a on a MacBook Pro with 16G memory and 3.3GHz Intel Core i7 processor.

Other parameters are chosen as follows
\begin{equation}\label{parameters}
\mu = 1/16, \quad \lambda = 1/4, \quad \rho_0 = 1, \quad \rho_1 = 4.
\end{equation}
The relative error is defined as
\[
E_{i+1} = \frac{|\Lambda_{i+1}-\Lambda_i|}{|\Lambda_i|},\quad i=1,2,3,4,
\]
where $\Lambda_{i}$ is the eigenvalue computed using the mesh with size $h_i$.
Then the convergence order is simply
\begin{align}
\mbox{convergence order}=\log_2{\frac{E_{i+1}} {E_{i+2}}},\quad \ i=1,2,3.
\end{align}

We present the results of the first several transmission eigenvalues.
Table~\ref{tablemix1} gives the computed eigenvalues and the convergence orders of the first real transmission eigenvalues of three
domains using the mixed method. It can be seen that the convergence rate for the unit square is approximately
2 indicating that the associated eigenfunction $u\in H^3(D)$. The convergence rate for the L-shaped domain is
lower, which is likely caused by the low regularity of the eigenfunction. Similar results can be observed
for the biharmonic eigenvalue problem (see Chap. 4 of \cite{SunZhou2016}). These results are consistent with the results in \cite{Ji2017} by noting that $\omega^2$ are given in \cite{Ji2017}.
 Table~\ref{tablemix2} gives the second real eigenvalues and convergence orders of three domains.  Table~\ref{tablemixc} gives the first complex eigenvalues.

\begin{table}
\begin{center}
\begin{tabular}{lllllllllll}
\hline
$h$&Unit square &order & L-shaped &order& Circle &order \\
\hline
0.1&1.547133 & & 2.667934 & &1.653707 & \\
0.05&1.428624& & 2.338044& &1.490512 & \\
0.025&1.402599& 2.072056 &  2.242452&1.596602 &1.461834& 2.358685\\
0.0125&1.396056&1.965350 &  2.215263& 1.753641&  1.453936 & 1.832356\\
0.00625&1.394419&1.992152 &  2.207749& 1.837771&  1.451948 & 1.982354\\
\hline
\end{tabular}
\caption{The first real transmission eigenvalue of the mixed method $\mu = 1/16, \lambda = 1/4, \rho_0 = 1, \rho_1 = 4$.}
\label{tablemix1}
\end{center}
\end{table}

 \begin{table}
\begin{center}
\begin{tabular}{lllllllllll}
\hline
$h$&Unit square &order & L-shaped &order& Circle &order \\
\hline
0.1&1.797671& & 2.818168 & &1.885692 & \\
0.05&1.661963& & 2.433875& &1.746810 & \\
0.025&1.629471& 1.949109 &  2.333153&1.720325&1.716030& 2.063423\\
0.0125&1.621129&1.933135 &  2.307061& 1.887725&  1.707564 & 1.836593\\
0.00625&1.619008&1.968244 &  2.300660& 2.011014&  1.705370 & 1.940982\\
\hline
\end{tabular}
\caption{The second real transmission eigenvalue of the mixed method $\mu = 1/16, \lambda = 1/4, \rho_0 = 1, \rho_1 = 4$.}
\label{tablemix2}
\end{center}
\end{table}

\begin{table}
\begin{center}
\begin{tabular}{lllllllllll}
\hline
$h$&Unit square &order & L-shaped &order& Circle &order \\
\hline
0.1&1.959412 - 0.287003i& & 2.068887 - 0.805506i & &2.048788 - 0.210752i & \\
0.05&1.892434 - 0.295354i& & 2.048189 - 0.760764i& &2.010480 - 0.280115i  & \\
0.025&1.873158 - 0.292942i & 1.748678 &  2.043123 - 0.749227i &1.944984 &1.994452 - 0.283584i&  2.251690\\
0.0125&1.867646 - 0.291971i&1.780619  &  2.041622 - 0.746321i &  1.939748&  1.989228 - 0.283208i
& 1.635355\\
0.00625&1.866145 - 0.291760i&1.880317 &  2.041219 - 0.745669i& 2.092299&  1.987713 - 0.283122i & 1.784153\\
\hline
\end{tabular}
\caption{The first complex transmission eigenvalue of the mixed method $\mu = 1/16, \lambda = 1/4, \rho_0 = 1, \rho_1 = 4$.}
\label{tablemixc}
\end{center}
\end{table}

From the Appendix, a radially-symmetric transmission eigenvalue of the disk is the first root of $Z_0$ defined in \eqref{eqZ0}.
Using some root finding technique, we find the smallest root $\omega=3.554954$. However, it is not the smallest transmission eigenvalue of the disk.
The mixed method also computes the transmission eigenvalues $\omega=3.555618$ with $h=0.00625$, $\omega=3.557610$ with $h=0.0125$. Convergence order is $2$.
Figure \ref{eigfunction} plots the eigenfunction ${\boldsymbol u}$ associated with this eigenvalue, which appear to be radially-symmetric.
\begin{figure}
\centering
\subfigure{\includegraphics[width=0.3\textwidth,height=0.25\textwidth]{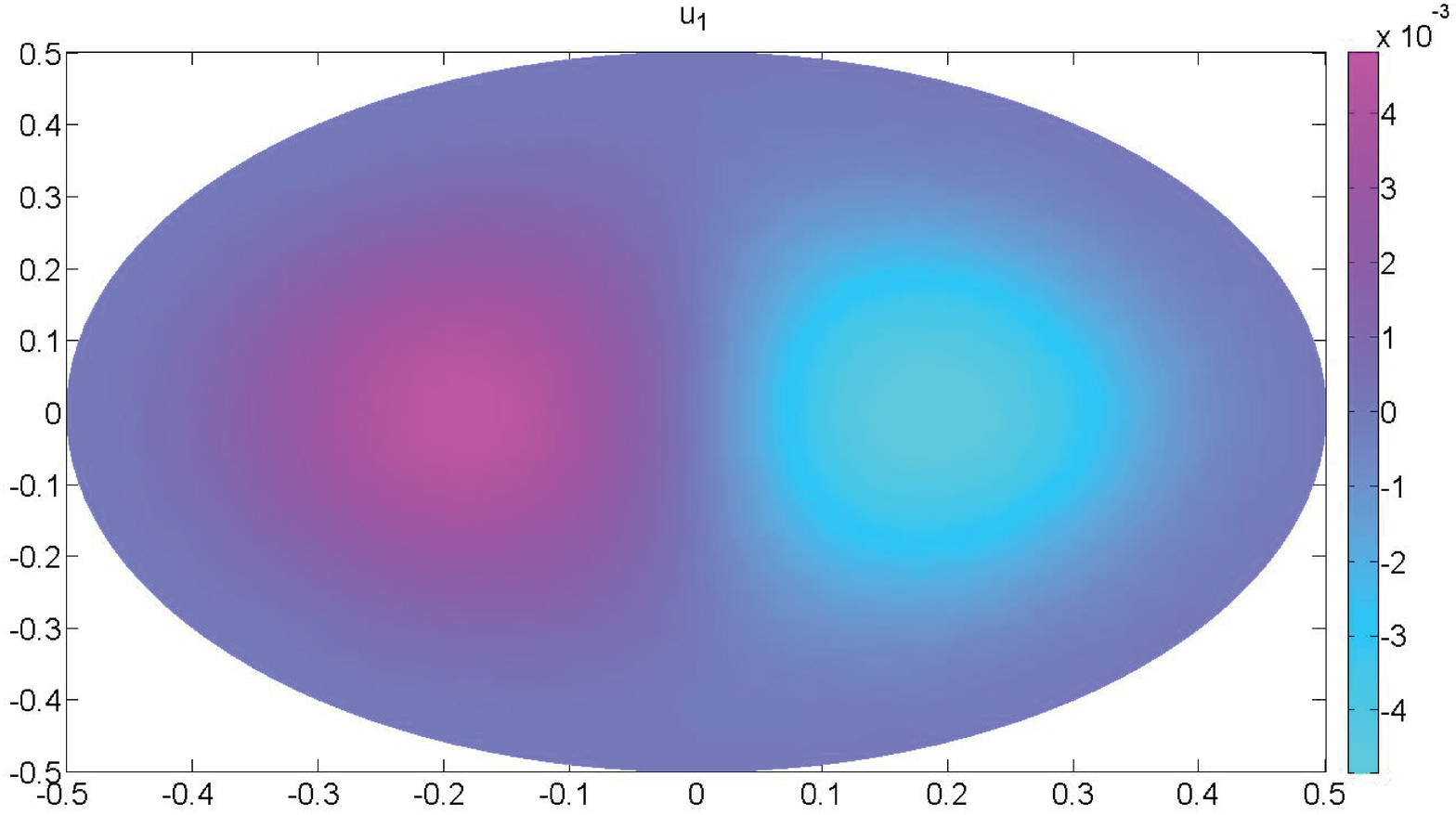}}
\subfigure{\includegraphics[width=0.3\textwidth,height=0.25\textwidth]{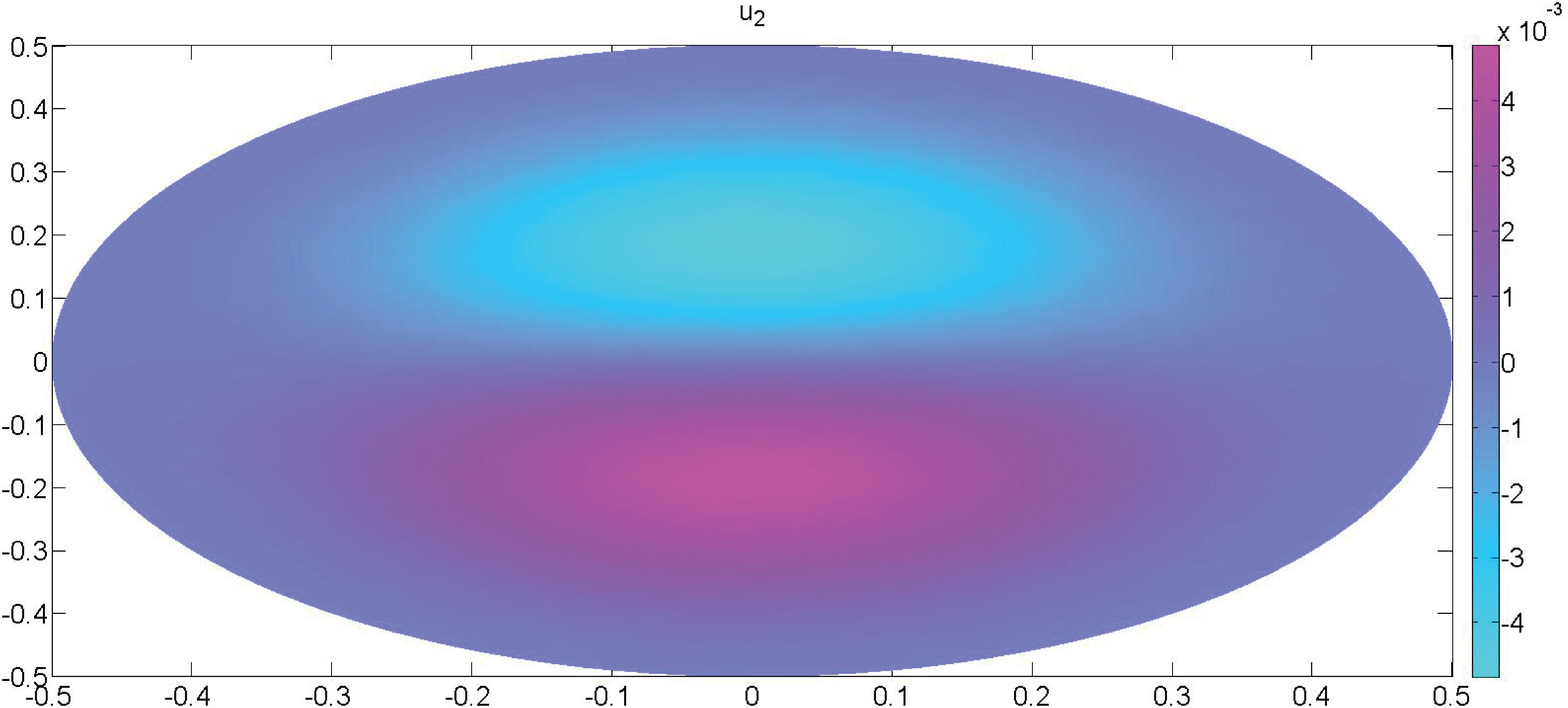}}
\subfigure{\includegraphics[width=0.3\textwidth,height=0.25\textwidth]{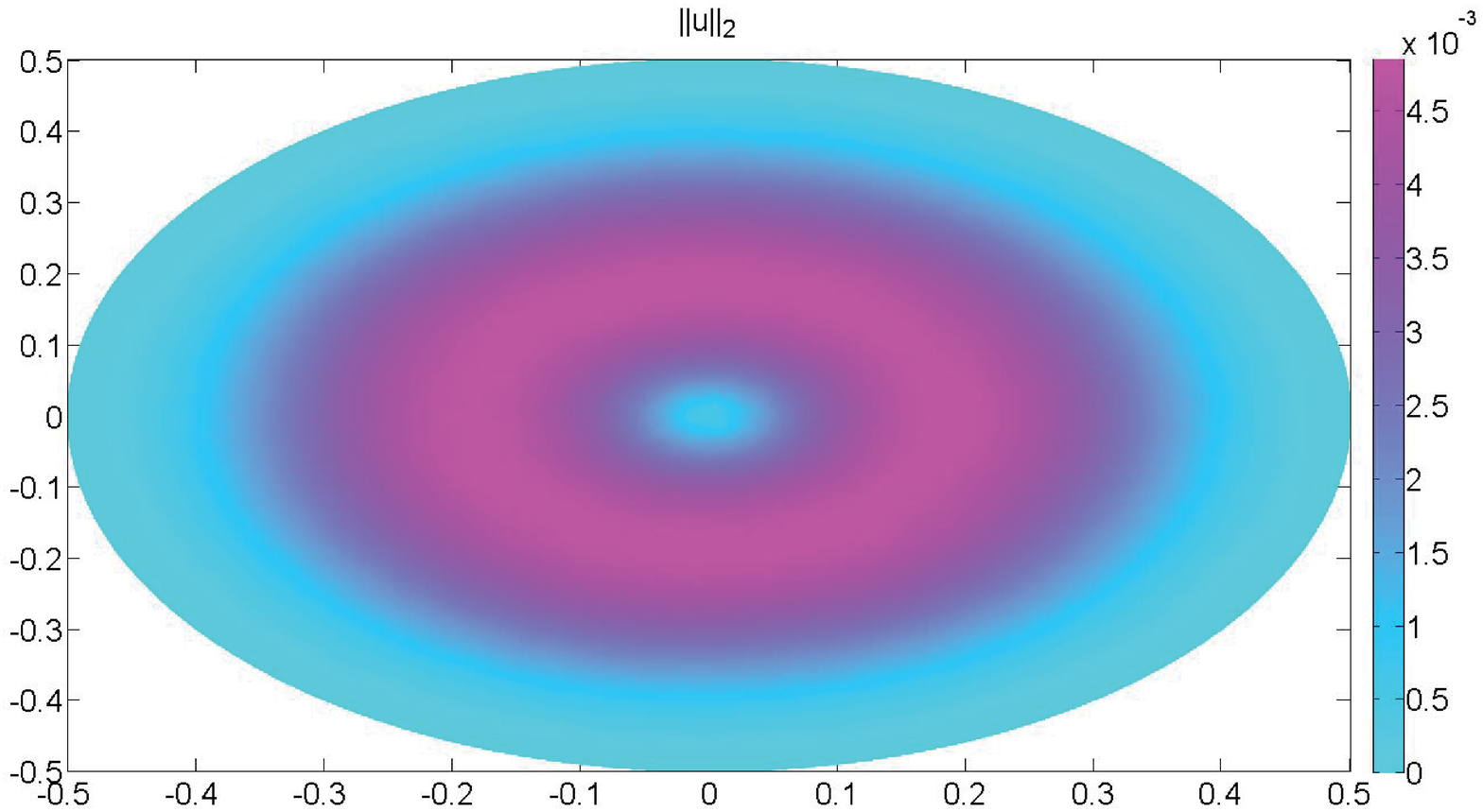}}
\caption{A radially-symmetric eigenfunction. Left: $u_1$. Middle: $u_2$. Right: $|{\boldsymbol u}=(u_1,u_2)|$.}
\label{eigfunction}
\end{figure}
Note that not all eigenfunctions are radially-symmetric. Figure \ref{secondeig} is the eigenfunction associated with the second eigenvalue. Clearly, it is not a radially-symmetric function.
\begin{figure}
\centering
\subfigure{\includegraphics[width=0.3\textwidth,height=0.25\textwidth]{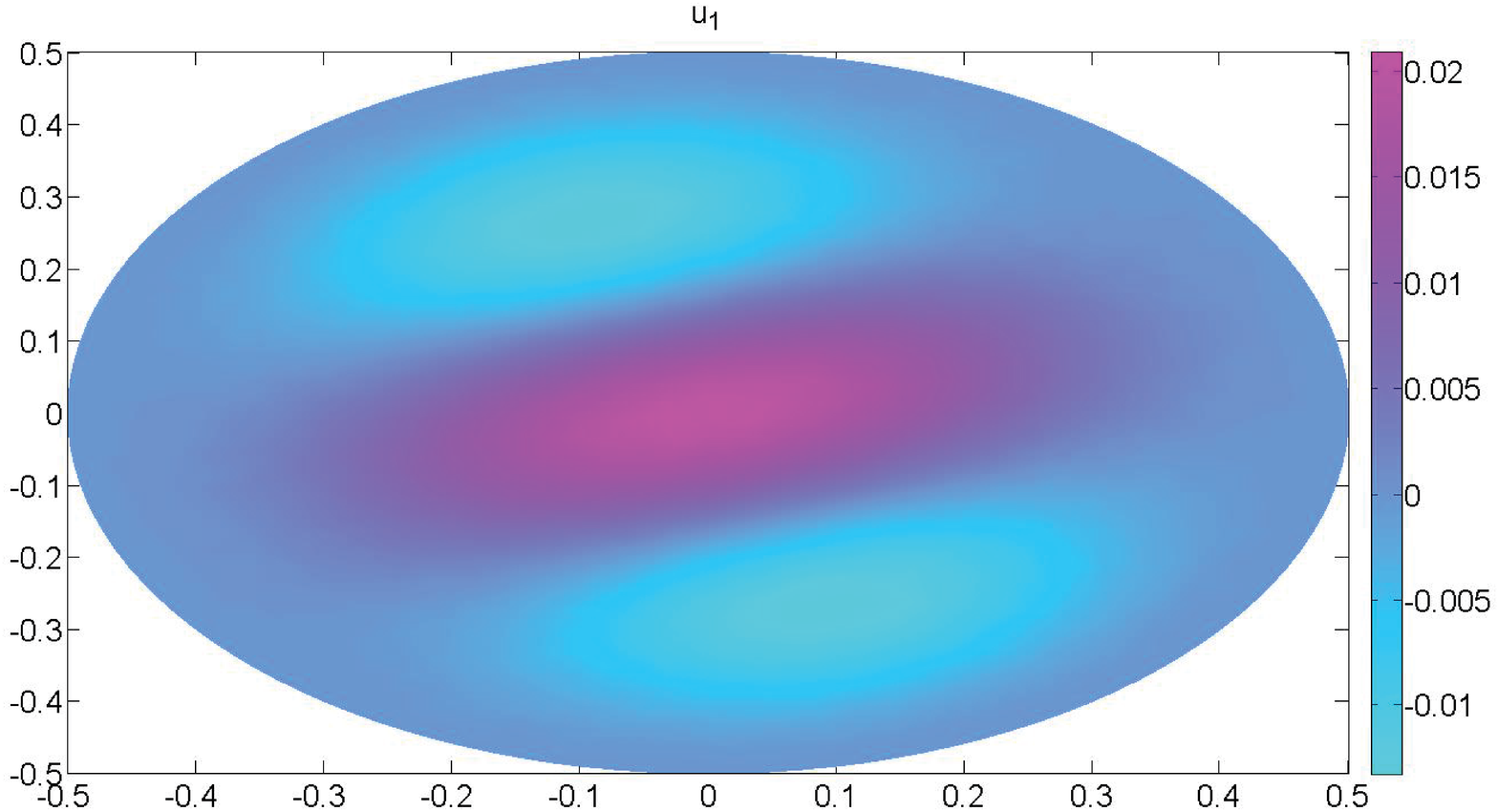}}
\subfigure{\includegraphics[width=0.3\textwidth,height=0.25\textwidth]{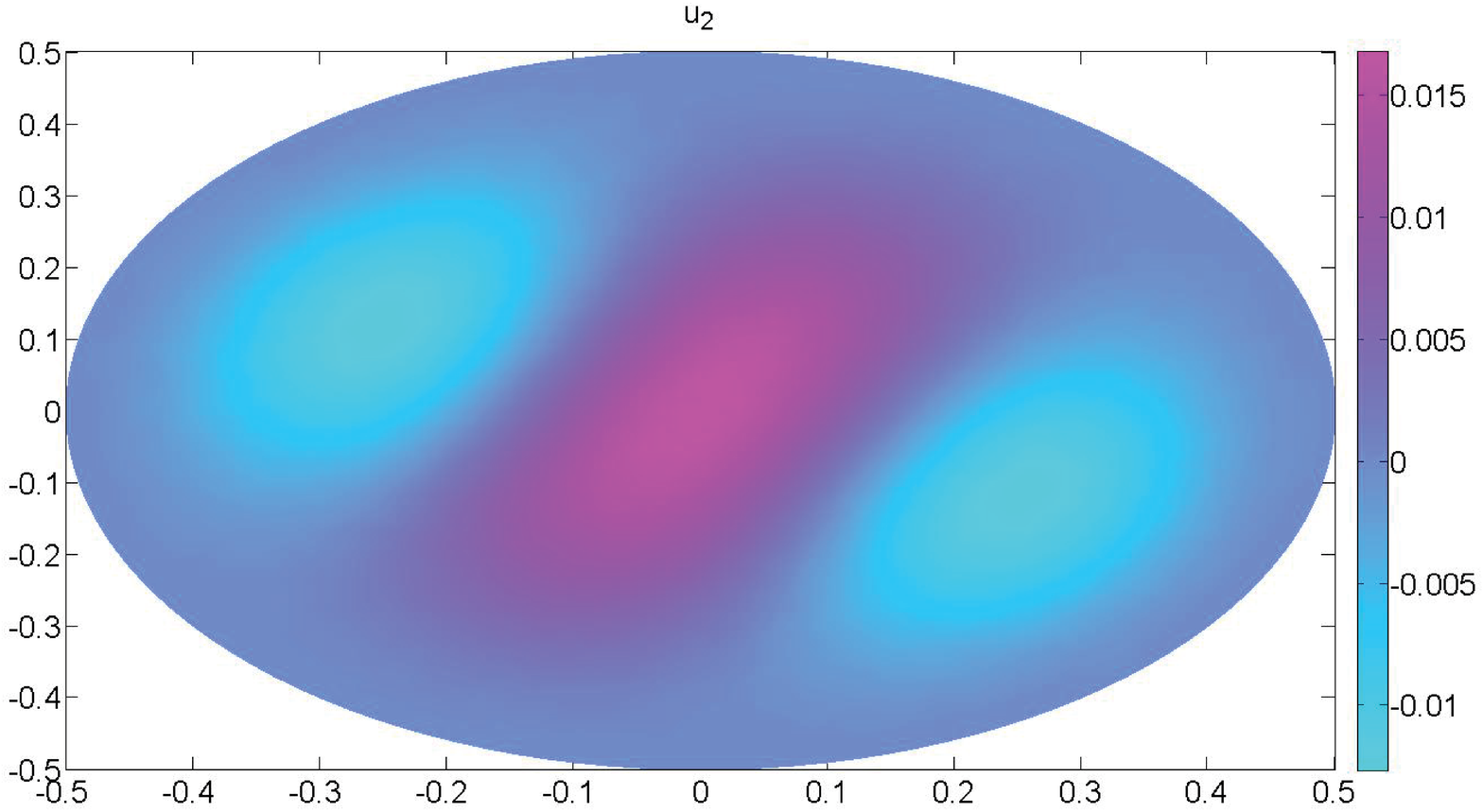}}
\subfigure{\includegraphics[width=0.3\textwidth,height=0.25\textwidth]{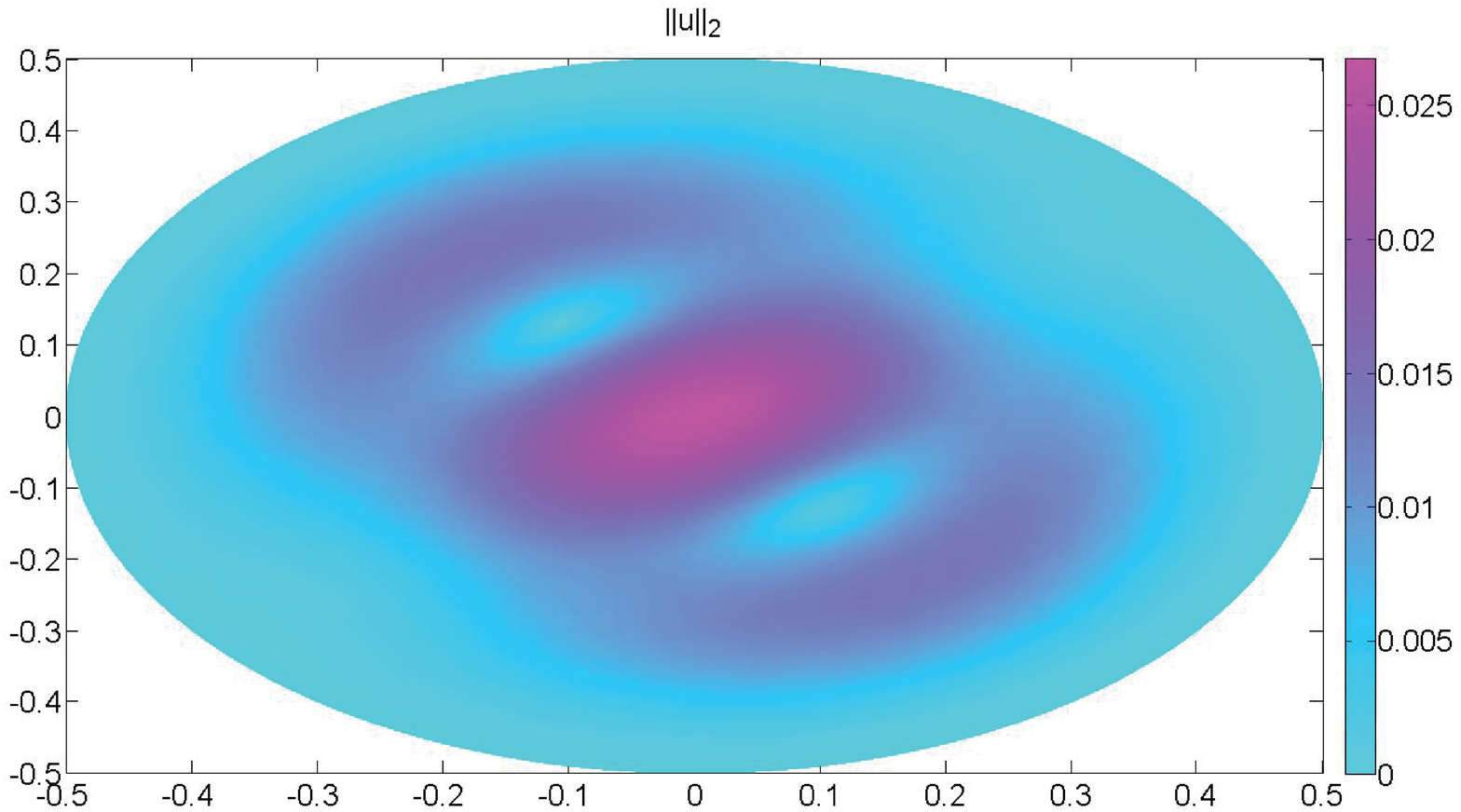}}
\caption{Second eigenfunction. Left: $u_1$. Middle: $u_2$. Right: $|{\boldsymbol u}=(u_1,u_2)|$.}
\label{secondeig}
\end{figure}

We also test the parameters
\begin{equation}\label{parameters1}
\mu = 1/4, \quad \lambda = 1/4, \quad \rho_0 = 1/20, \quad \rho_1 = 3.
\end{equation}
Table~\ref{tablefirstten} gives the first ten real eigenvalues of three domains, which is consistent with the result of $\omega^2$ in \cite{XiJi2018}. We also test the convergence order of the first
real eigenvalues, the results are given in Table~\ref{tablepara2}.

\begin{table}
\begin{center}
\begin{tabular}{cccc}
\hline
Eigenvalue & Unit square &L-shaped    &Circle     \\
\hline
$\Lambda_1$ &2.840221              &3.681961      &2.978253 \\
$\Lambda_2$ &3.092466             &4.132549      &3.359621 \\
$\Lambda_3$ &3.092482               &4.551687      &3.359652 \\
$\Lambda_4$ &3.742555              &4.768572      &3.988719\\
$\Lambda_5$ &3.742593             &4.941573      &3.988775\\
$\Lambda_6$ &3.776228              &5.153489   &4.207261\\
$\Lambda_7$ &3.890075              &5.171420      &4.207329\\
$\Lambda_8$ &4.636315              &5.297215    &4.883177\\
$\Lambda_9$ &4.538839              &5.340098     &4.943635\\
$\Lambda_{10}$&4.538879              &5.719738      &4.943363\\
\hline
\end{tabular}
\end{center}
\caption{The first ten transmission eigenvalues with $h\approx0.00625$ and $\mu = 1/4, \lambda = 1/4, \rho_0 = 1/20, \rho_1 = 3$.}
\label{tablefirstten}
\end{table}

\begin{table}
\begin{center}
\begin{tabular}{lllllllllll}
\hline
$h$&Unit square &order & L-shaped &order& Circle &order \\
\hline
0.1&2.943315& & 3.825626 & &3.090077 & \\
0.05&2.866060& & 3.718472& &3.006292 & \\
0.025&2.846493& 1.942833&  3.690525&1.897934&2.985290&1.956508\\
0.0125&2.841483&1.955657 &  3.683535& 1.988445&  2.979681 & 1.894597\\
0.00625&2.840221&1.986557 &  3.681961& 2.148121&  2.978253 & 1.971034\\
\hline
\end{tabular}
\caption{The first real transmission eigenvalue of the mixed method $\mu = 1/4, \lambda = 1/4, \rho_0 = 1/20, \rho_1 = 3$.}
\label{tablepara2}
\end{center}
\end{table}


\section*{Appendix: Radially Symmetric Case on Disks}

We derive the equation satisfied by a transmission eigenvalue whose associated eigenfunction is radially symmetric on a disk.
Let $D \subset \mathbb R^2$ be a disk with radius $R$.
Let ${\boldsymbol u}=(w, v)^\top$. Writing the elastic wave equation
\eqref{ElasticityLHS1}
component wise, we have that
\begin{eqnarray} \label{2mulambda}
(2\mu + \lambda) \frac{\partial^2 w}{\partial x_1^2}+(\lambda + \mu)\frac{\partial^2 v}{\partial x_2 \partial x_1}
				+ \mu \frac{\partial^2 w}{\partial^2 x_2}+\omega^2 \rho w = 0, \\
\label{mulambda}	   (\mu+\lambda)\frac{\partial^2 w}{\partial x_2 \partial x_1}+\mu \frac{\partial^2 v}{\partial x_1^2} +
	   (2\mu+\lambda) \frac{\partial^2 v}{\partial x_2^2}+\omega^2 \rho v = 0.
\end{eqnarray}
If we consider the solution in the form of radially-symmetric vector field
${\boldsymbol u}({\boldsymbol x}) = u(r){\boldsymbol e}_r$, where
$r=|{\boldsymbol x}|$ and ${\boldsymbol e}_r = {\boldsymbol x}/{r}$, $w = u(r)\cos \theta, v = u(r)\sin \theta$,
\eqref{2mulambda} can be written as
\[
(\mu + \lambda) \left( \frac{\partial^2 w}{\partial x_1^2}+\frac{\partial^2 v}{\partial x_2 \partial x_1}\right)
				+ \mu \left(\frac{\partial^2 w}{\partial^2 x_2}+ \frac{\partial^2 w}{\partial x_1^2}\right) +\omega^2 \rho w =0.
\]
Using polar coordinate, \eqref{2mulambda} becomes
\[
(\lambda+\mu) \left( u_{rr}  +\frac{1}{r} u_r  -\frac{1}{r^2} u   \right)
+\mu\left(u_{rr} + \frac{1}{r} u_r - \frac{1}{r^2}u\right) +\omega^2 \rho u = 0,
\]
i.e.,
\[
(\lambda+2\mu) \left( u_{rr}  +\frac{1}{r} u_r  -\frac{1}{r^2} u
\right)+\omega^2 \rho u = 0.
\]
Similarly, \eqref{mulambda} is simply
\[
(\mu+\lambda)\left(u_{rr} + \frac{1}{r}u_r -\frac{1}{r^2} u\right)+
			\mu \left(u_{rr} + \frac{1}{r}u_r - \frac{1}{r^2}u\right)
	  +\omega^2 \rho u = 0,
\]
i.e.,
\[
(2\mu+\lambda)\left(u_{rr} + \frac{1}{r}u_r -\frac{1}{r^2} u\right) +\omega^2 \rho u = 0.
\]
The above equation can be written as
\[
r^2 \frac{{\rm d}^2 u}{{\rm d}r^2}+r \frac{{\rm d}u}{{\rm d}r} + \left( r^2
\frac{\omega^2 \rho}{2\mu+\lambda}-1 \right) u = 0.
\]
The solution of the above equation is given by the Bessel function of order one
$J_1(ar)$, where
\[
a = \omega \sqrt{\frac{\rho}{2\mu + \lambda}}.
\]
Then we obtain that ${\boldsymbol u}=(w, v)^\top:=(J_1(ar)\cos \theta,
J_1(ar) \sin \theta)^\top$.

Next we look at the boundary condition involving $\sigma({\boldsymbol u}) {\boldsymbol \nu}$. For the
transmission eigenvalue problem, we assume that
\[
{\boldsymbol u} = \begin{pmatrix} J_1(a_1 r) \cos \theta \\ J_1(a_1 r) \sin \theta \end{pmatrix}, \quad
{\boldsymbol v} = C\begin{pmatrix} J_1(a_2 r) \cos \theta \\ J_1(a_2 r) \sin \theta \end{pmatrix},
\]
where $C$ is a constant to be determined from \eqref{tep} and
\[
a_1 = \omega \sqrt{\frac{\rho_1}{2\mu + \lambda}}, \quad a_2 = \omega \sqrt{\frac{\rho_2}{2\mu + \lambda}}.
\]

Note that ${\boldsymbol \nu} = (\nu_1, \nu_2)^\top = (\cos \theta, \sin
\theta)^T$. It follows from \eqref{DefSigma} that the first component of
$\sigma({\boldsymbol u}){\boldsymbol \nu}$ is
\begin{eqnarray*}
\lambda \left[ \frac{\partial }{\partial r}J(ar) + \frac{1}{r} J(ar) \right]
\cos \theta  + 2\mu \frac{\partial }{\partial r}J(ar) \cos \theta.
\end{eqnarray*}
The second component of $\sigma({\boldsymbol u}){\boldsymbol \nu}$ is
\begin{eqnarray*}
 2\mu \frac{\partial }{\partial r}J(ar)  \sin \theta + \lambda \left[
\frac{\partial }{\partial r}J(ar) + \frac{1}{r} J(ar) \right] \sin \theta.
\end{eqnarray*}

Using the boundary conditions \eqref{tep}, we obtain
\[
C=\frac{J_1(a_1 R)}{J_1(a_2 R)},
\]
and
\[
\left(2\mu \frac{\partial }{\partial r}J(a_1r)  + \lambda \left[  \frac{\partial }{\partial r}J(a_1r) + \frac{1}{r} J(a_1r) \right] \right) \Big\vert_{r=R}
= C\left(2\mu \frac{\partial }{\partial r}J(a_2r)  + \lambda \left[  \frac{\partial }{\partial r}J(a_2r) + \frac{1}{r} J(a_2r) \right] \right) \Big \vert_{r=R}.
\]
Hence $\omega$ is a transmission eigenvalue if it satisfies
\begin{equation}\label{eqZ0}
Z_0(\omega):=\left | \begin{array}{cc} J_1\left( a_1R\right) & J_1\left( a_2 R\right)  \\
a_1J_1'(a_1 R)  & a_2 J_1'(a_2 R)  \end{array}\right|=0.
\end{equation}
Figure \ref{absZ0s} is the contour plot of $|Z_0(\omega)|$ on the complex plane.

\begin{figure}
\begin{center}
{ \scalebox{0.3} {\includegraphics{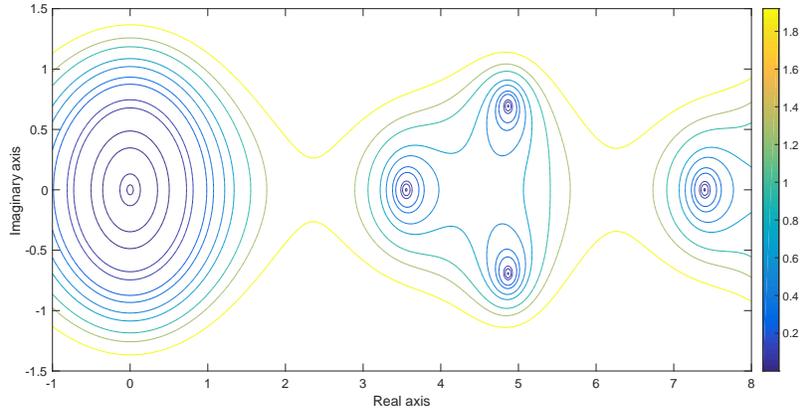}}}
\caption{The contour plot of $|Z_0(\omega)|$ with $\mu = 1/16, \lambda = 1/4,
\rho_0 = 1, \rho_1 = 4$. The centers of the circular curves indicate the
locations of transmission eigenvalues.}
 \label{absZ0s}
\end{center}
\end{figure}


\end{document}